\documentclass[a4paper,twoside,12pt]{article}
\usepackage[english]{babel}
\usepackage{bbding}
\usepackage[latin1]{inputenc}
\usepackage[T1]{fontenc}
\usepackage{kpfonts}
\usepackage{amsmath,amssymb,color,epsfig}
\usepackage{mathrsfs}
\usepackage{eufrak}
\usepackage{dsfont}
\usepackage{upgreek}
\usepackage{fancyhdr}
\usepackage{amsthm}
\newtheorem{theorem}{Theorem}[section]
\newtheorem{proposition}[theorem]{Proposition}
\newtheorem{corollary}[theorem]{Corollary}

\newtheorem{lemma}[theorem]{Lemma}

\def\1{\mathds{1}}

\newcommand{\footremember}[2]{%
   \footnote{#2}
    \newcounter{#1}
    \setcounter{#1}{\value{footnote}}%
}

\title{Convexity and concavity of a class of functions related to the elliptic functions}
\author{%
    Mohamed Bouali\footremember{alley}{Department of mathematics, Preparatory Institute for Engineering Studies of Tunis\newline
     Department of mathematics, Faculty of Sciences of Tunis}%
}

\date{}
\begin{document}

\maketitle
\begin{abstract} We investigate the convexity property on $(0,1)$ of the function $$f_a(x)=\frac{{\cal K}{(\sqrt x)}}{a-(1/2)\log(1-x)}.$$
We show that $f_a$ is strictly convex on $(0,1)$ if and only if $a\geq a_c$ and $1/f_a$ is strictly convex on $(0,1)$ if and only if $a\leq\log 4$, where $a_c$ is some critical value. The second main result of the paper is to study the log-convexity and log-concavity of the function $$h_p(x)=(1-x)^p{\cal K}(\sqrt x).$$
We prove that $h_p$ is strictly log-concave on $(0,1)$ if and only if $p\geq 7/32$ and strictly log-convex if and only if $p\leq 0$. This solves some problems posed by Yang and Tian and complete their result and a result of Alzer and  Richards that $f_a$ is strictly concave on $(0,1)$ if and only if $a=4/3$ and $1/f_a$ is strictly concave on $(0,1)$ if and only if $a\geq 8/5$. As applications of the convexity and concavity, we establish among other inequalities, that for $a\geq a_c$ and all $r\in(0,1)$
$$\frac{2\pi\sqrt\pi}{(2a+\log 2)\Gamma(3/4)^2}\leq \frac{{\cal K}(\sqrt r)}{a-\frac12\log (r)}+\frac{{\cal K}(\sqrt{1-r})}{a-\frac12\log (1-r)}<1+\frac\pi{2a},$$
and for $p\geq 3(2+\sqrt 2)/8$ and all $r\in(0,1)$
$$\sqrt{(r-r^2)^p{\cal K}(\sqrt{1-r}){\cal K}(\sqrt r)}< \frac{\pi\sqrt\pi}{2^{p+1}\Gamma(3/4)^2}<\frac{r^p{\cal K}(\sqrt{1-r})+(1-r)^p{\cal K}(\sqrt r)}{2}.$$
\end{abstract}




\maketitle

\section{Introduction and statement of the results}\label{sec1}

The complete elliptic integral of the first kind is defined on $[0,1)$ by
$${\cal K}(r)=\int_0^{\pi/2}\frac{dt}{\sqrt{1-r^2\sin^2(t)}}=\frac\pi2{}_2F_1(1/2,1/2,1,r^2),$$
and the complete  elliptic integral of the second kind
$${\cal E}(r)=\int_0^{\pi/2}\sqrt{1-r^2\sin^2(t)}dt=\frac\pi2{}_2F_1(1/2,-1/2,1,r^2),$$
where ${}_2F_1$ is the Gaussian hypergeometric function
$${}_2F_1(a,b,c,x)=\sum_{n=0}^\infty\frac{(a)_n(b)_n}{(c)_n}\frac{x^n}{n!},\quad (-1<x<1),$$
and $(a)_n=a(a+1)\cdots (a+n-1)=\Gamma(a+n)/\Gamma(a)$ the Pochhammer symbol.

Special functions and especially elliptic functions arise in numerous branches of mathematics such as geometric function theory
and quasi-conformal mappings, also in physics, theory of mean values, number theory and other related fields, see for instance
\cite{byr, ber, bal, bha, ande, qiu, wang, alz}. Many authors were interested in studying convexity and concavity properties of functions related to $\cal  K$ and $\cal E$. Anderson et al. in \cite{ander} showed that the function defined on $(0,1)$ by
$$U_a(x)=\frac{{\cal K}(x)}{a-(1/2)\log(1-x^2)},$$
is strictly decreasing if and only if $0\leq b\leq\log 4$, and strictly increasing if and only if $b\geq 2$. In a recent paper \cite{yan}, Yang and Tian  studied the closely related function
$$V_b(x)=\frac{{\cal K}(\sqrt x)}{b-(1/2)\log(1-x)},\quad c\geq 0.$$
They proved that $V_b$ is strictly concave on $(0,1)$ if and only if $b=4/3$. They also posed the following problem.

{\it Determine the best parameters $a$ and $b$ such that $V_a$ is convex on $(0, 1)$ and $1/V_b$ is concave on $(0, 1)$.}

In the recent paper \cite{alz1}, Alzer and Richards give an answer to the second problem. However, the first problem remains open until now.

Later in \cite{rich}, Richards and Smith extended the second problem to the generalized elliptic integral ${\cal K}_p, (p\geq 1)$.

The primary objective of this paper is to address the first part of the problem and to give an answer. Additionally, we study the convexity property of the function $1/V_b$.

In the same paper \cite{yan}, the authors conjectured that:

 {\it The function $h_p(x)=(1-x)^p {\cal K}(\sqrt x)$ is log-concave on $(0,1)$ if and only if $p\geq 7/32$.}

  It is our aim in the second part of the paper to solve this problem. After writing this paper, the author discovers that the log-convexity of the function $h_p$ has been proved by Wang et all in \cite{wan} by another method.

  In what follows, we adopt these notations $K(x)={\cal K}(\sqrt x)$, $E(x)={\cal E}(\sqrt x)$
\begin{theorem}\label{t2} For $a\in\Bbb R$, the function
$$f_a(x)=\frac{ K( x)}{a-\frac12\log(1-x)},$$
is strictly convex on $(0,1)$ if and only if $a\geq a_c$ and is concave on $(0,1)$ if and only if $a=4/3$.

Where $$a_c=\max_{x\in(0,1)} \Big[\frac12\log(1-x)+\frac{v(x)+\sqrt{\Delta(x)}}{2 u(x)}\Big],$$
$$u(x)=\frac1{16}{}_2F_1(3/2,3/2,3,x)(1-x)+\frac12{}_2F_1(1/2,1/2,2,x),$$
$$v(x)=\frac12{}_2F_1(1/2,1/2,2,x)+{}_2F_1(1/2,1/2,1,x),$$
 and $$\Delta(x)=v(x)^2-4u(x){}_2F_1(1/2,1/2,1,x).$$
Numerically $a_c\simeq1.4622$.
\end{theorem}
The constant $a_c$ is sharp. In the sense, it can not be replaced by a constant less than $a_c$.
\begin{theorem}\label{t3} For $a\in\Bbb R$, the function
$$\frac1{f_a(x)}=\frac{a-\frac12\log(1-x)}{ K( x)},$$
is strictly convex on $(0,1)$ if and only if $ a\leq \log 4$ and strictly concave on $(0,1)$ if and only if $a\geq 8/5$.
\end{theorem}
The constant $\log 4$ and $8/5$ are sharp. In the sense, they can not be replaced by constants bigger than $\log 4$ and less than $8/5$.
 \begin{theorem}\label{t1} For $p\in\Bbb R$, the function $h_p(x)=(1-x)^pK(x)$ is log-concave on $(0,1)$ if and only if $p\geq 7/32$ and log-convex if and only if $p\leq 0$.
 \end{theorem}
 The second part of Theorem \ref{t2} and Theorem \ref{t3} is proved respectively in \cite{alz1} and in \cite{yan}.

 Consequently, we have the following corollaries
 \begin{corollary}\label{p1} The function $h_p(x)=(1-x)^pK(x)$ is strictly convex on $(0,1)$ if and only if one of the following conditions  hold $p\leq 0$ or $p\geq 3(2+\sqrt 2)/8$ and strictly concave if and only if $p\in[ 3(2-\sqrt 2)/8,1]$.
\end{corollary}
\begin{corollary}\label{p}
The function $h_p(x)=(1-x)^pK(x)$ is strictly decreasing on $(0,1)$ if and only if $p\geq 1/4$, and strictly increasing if and only if $p\leq 0$.

If $p\in(0,1/4)$ there exists a unique $x_p\in(0,1)$ such that $h_p$ is strictly increasing on $(0,x_p)$ and strictly decreasing on $(x_p,1)$.

\end{corollary}
\section{Preliminary results}
In this section, we collect some results which are needed to prove Theorem \ref{t2}, Theorem \ref{t3} and Theorem \ref{t1}. The first lemma offers three basic properties of the hypergeometric function ${}_2F_1$ (see [\cite{olv}, 15.4.20, 15.4.21, 15.5.1,
15.8.1]) and \cite{ask}.
\begin{lemma}\label{lem1} For $x\in(-1,1)$
$$\frac{d}{dx}{}_2F_1(a,b;c,x)=\frac{ab}c{}_2F_1(a+1,b+1;c+1,x),$$
$${}_2F_1(a,b;c,x)=(1-x)^{c-b-a}{}_2F_1(c-a,c-b;c,x),$$
$$\lim_{x\to 1^-}\frac{{}_2F_1(a,b;a+b,x)}{-\log(1-x)}=\frac{\Gamma(a+b)}{\Gamma(a)\Gamma(b)},$$
and if $c>a+b$,
$${}_2F_1(a,b;c,1)=\frac{\Gamma(c)\Gamma(c-b-a)}{\Gamma(c-b)\Gamma(c-a)}.$$

For $x\to 1$
$$K(x)=\log 4-\theta(x)-\frac14(1-x)\theta(x)+o((1-x)\theta(x)),$$
where $\theta(x)=(-1/2)\log(1-x)$.
\end{lemma}
 We Recall the following two lemmas, see for instance  ([\cite{alz}, Theorem 15] and [\cite{vuo}, Theorem 3.21] and \cite{vam}.
\begin{lemma}\label{r3}
\begin{enumerate}
\item The function $x\mapsto (E(x)-(1-x)K(x))/x$ is strictly increasing from $(0,1)$ onto $(\pi/4,1)$.
\item The function $x\mapsto (E^2(x)-(1-x)K^2(x))/x^2$ is strictly increasing from $(0,1)$ onto $(\pi^2/32,1)$.
 \end{enumerate}
\end{lemma}
\begin{lemma}\label{r4} Let $g$ and $h$ be real-valued functions, which are continuous on $[a, b]$ and differentiable
on $(a, b)$. Further, let $h'(x)\neq 0$ on $(a, b)$. If $g'/h'$ is strictly increasing (resp. decreasing) on $(a, b)$,
then the functions
$$x\mapsto (g(x)-g(a))/(f(x)-f(a))\;{\rm and}\quad x\mapsto (g(x)-g(b))/(f(x)-f(b)),$$
are also strictly increasing (resp. decreasing) on $(a,b)$.
\end{lemma}
In the proofs of the next results, we leverage the following differentiation formulas
$$\frac{d}{dx}K(x)=\frac{E(x)-(1-x)K(x)}{2x(1-x)},$$
$$\frac{d}{dx}E(x)=\frac{E(x)-K(x)}{2x}.$$
\begin{lemma}\label{r} \
\begin{enumerate}
\item The function $x\mapsto (2-x)K(x)-2E(x)$ is strictly increasing from $(0,1)$ onto $(0,+\infty)$.
\item For $0<x<1$, $$(2- x)K(x)-2E(x)\geq \frac2\pi(E(x)^2-(1-x)K(x)^2).$$
\item The function $x\mapsto E(x)+\sqrt{1-x} K(x)$ is strictly decreasing from $(0,1)$ onto $(1,\pi)$.
\item The function $x\mapsto \big(E(x)^2-(1-x)K(x)^2\big)/(x^2K(x))$ is strictly increasing on $(0,\alpha)$, and
    $$E(x)^2-(1-x)K(x)^2\geq\frac{\pi}{16} x^2K(x),$$
    where $\alpha=(8/97)(11-2\sqrt 6)$.
    \item The function $x\mapsto \big(E(x)-(1-x)K(x)\big)/(x^2K(x))$ is strictly increasing on $(0,\alpha)$, and
    $$E(x)-\sqrt{1-x}K(x)\geq\frac{x^2}{16} K(x).$$
    \item The function $$\varphi(x)=\frac12\log(1-x)+\frac{2xK(x)(E(x)-K(x))}{2E(x)^2-2E(x)K(x)+x(1-x)K(x)^2},$$
    is strictly decreasing from $(0,1)$ onto $(\log 4,8/5)$.
In particular, the double inequality holds
$$\frac{2xK(x)(K(x)-E(x))}{8/5-1/2\log(1-x)}\leq2E(x)K(x)- 2E(x)^2-x(1-x)K(x)^2\leq\frac{2xK(x)(K(x)-E(x))}{\log4-1/2\log(1-x)}.$$
  \end{enumerate}
\end{lemma}
\begin{proof} \
1) Let $h_1(x)=(2-x)K(x)-2E(x)$. A direct computation gives $h'(x)=(E(x)-(1-x) K(x))/(2-2 x)$ which is positive by Lemma \ref{r3}, moreover, $h(0)=0$ and $\displaystyle\lim_{x\to1^-1}h(x)=+\infty$.

2) For $x\in(0,1)$, set $$u(x)=\frac{E(x)^2-(1-x)K(x)^2}{(2- x)K(x)-2E(x)}=\frac{g(x)}{h(x)}.$$
We have, $$\frac{g'(x)}{h'(x)}=\frac{2(1-x)}x\frac{ (E(x)-K(x))^2}{(E(x)-(1-x)K(x))},$$
Performing  another differentiation gives
$$\Big(\frac{g'}{h'}\Big)'(x)=\frac{\big(K(x)-E(x)\big)\big(2 E(x)-(2-x)K(x)\big) \big((x+1) E(x)-(1-x) K(x)\big)}{x^2 (E(x)-(1-x)K(x))^2}.$$
From Lemma \ref{r3}, item 1) of Lemma \ref{r} and the inequality $K(x)>E(x)$ for $x\in(0,1)$, we deduce that the function $g'/h'$ is strictly decreasing. Since, $g(0)=0$ and $h(0)=0$, then by Lemma \ref{r4}, we obtain that the function $u=g/h$ is strictly decreasing on $(0,1)$. Moreover, From the series expansions of the elliptic functions, one gets

$h(x)=(\pi/16) x^2+o(x^2)$ and $g(x)=(\pi^2/32)x^2+o(x^2)$. Then, $u(x)\leq \pi/2$.

3) Differentiation yields,
$$\frac{d}{dx}(E(x)+\sqrt{1-x} K(x))=\frac{(\sqrt{1-x}+1) (E(x)-K(x))}{2 x\sqrt{1-x} }<0,$$
for $x\in(0,1)$. Since, $\displaystyle\lim_{x\to 1^-}\sqrt{1-x} K(x)=0$, whence,  $\pi=E(0)+K(0)\geq E(x)+\sqrt{1-x} K(x)\geq E(1)=1$.

4) Let us define, $$\frac{g(x)}{h(x)}=\frac{E(x)^2-(1-x)K(x)^2}{x^2K(x)},$$
then, $$\frac{g'(x)}{h'(x)}=\frac{2 (1-x) (E(x)-K(x))^2}{x^2 (E(x)+3 (1-x) K(x))},$$
and $$\Big(\frac{g'}{h'}\Big)'(x)=\frac{(E(x)-K(x)) \left((9 x^2-17 x+8) K(x)^2+x (5-3 x) K(x) E(x)+4 (x-2) E(x)^2\right)}{x^3 (E(x)-3 (x-1) K(x))^2}.$$
Since, $9 x^2-17 x+8=9(1-x)(8/9-x)>0$ for $x\in(0,8/9)$ and $\alpha<8/9$, from Lemma \ref{r3}, we obtain
$$\begin{aligned}&(9 x^2-17 x+8) K(x)^2+x (5-3 x) K(x) E(x)+4 (x-2) E(x)^2\\&\leq (8-9x)E(x)^2+x (5-3 x) K(x) E(x)+4(x-2) E(x)^2\\&=xE(x)(-5E(x)+(5-3x)K(x)).\end{aligned}$$
Let $F(x)=-5E(x)+(5-3x)K(x)$. Differentiate two times, we get
$F'(x)=\big(3 (x-1) K(x)+2 E(x)\big)/(2-2 x)$ and $2((1-x)F'(x))'=\big(3 x K(x)+K(x)-E(x)\big)/(2 x)>0.$ Then, $(1-x)F'(x)$ is strictly increasing on $(0,\alpha)$ and $(1-\alpha)F'(\alpha)= E(\alpha)-(3/2)(1-\alpha)K(\alpha)\simeq -0.0346906$.
 Then, $F(x)$ is strictly decreasing on $(0,\alpha)$ and $F(0)=0$. Therefore2,
$\Big(g'/h'\Big)'(x)>0$ on $(0,\alpha)$ and $g'/h'$ is strictly increasing, and by Lemma \ref{r4}, the function $g/h$ is strictly increasing on $(0,\alpha)$.
From the hypergeometric representation of the elliptic functions, we obtain
  $E(x)^2-(1-x)K(x)^2=(\pi^2/32)x^2+o(x^2)$ then, $\displaystyle\lim_{x\to 0}(g/h)(x)=\pi/16$.

5) On the one hand $$\frac{d}{dx}(E(x)+\sqrt{1-x}K(x))=\frac{(\sqrt{1-x}+1) (E(x)-K(x))}{2x\sqrt{1-x}}<0,$$
Therefore, the function $x\mapsto E(x)+\sqrt{1-x}K(x)$ is strictly decreasing and  positive on $(0,1)$. On the other hand by item 4) the function $x\mapsto (E(x)^2-(1-x)K(x)^2)/x^2K(x)$ is strictly increasing and positive on $(0,\alpha)$. Then, the function  $x\mapsto (E(x)-\sqrt{1-x}K(x))/x^2K(x)$ is strictly increasing on $(0,\alpha)$ and  $E(x)-\sqrt{1-x}K(x)=(\pi/32)x^2+o(x^2)$, then $\displaystyle\lim_{x\to 0}(E(x)-\sqrt{1-x}K(x))/x^2K(x)=1/16$.

6) Differentiation yields,
 $$\varphi'(x)=\frac{K(x)^2 \Theta_1(x)}{2 \left((x-1) x K(x)^2+2 K(x) E(x)-2 E(x)^2\right)^2},$$
 where
 $$\Theta_1(x)=\left(x^3-3 x^2-2 x+4\right) K(x)^2+8 (x-1) K(x) E(x)+(4-6 x) E(x)^2$$
 and $$\Theta'_1(x)=\left(2 x^2-3 x-4\right) K(x)^2+(16-x) K(x) E(x)-12 E(x)^2.$$
 Let $\Lambda(x)=E(x)/K(x)$, then
 $$\frac{\Theta'_1(x)}{K(x)^2}=\left(2 x^2-3 x-4\right) +(16-x) \Lambda(x)-12 \Lambda(x)^2.$$
 The discriminant of the polynomial $P(t)= \left(2 x^2-3 x-4\right) +(16-x) t-12 t^2$ is $D_P(x)=97 x^2-176 x+64$.

 By straightforward computation, we get  $D_P(x)\leq0$ if and only if $x\in[\alpha,\,\alpha+(32/97)\sqrt 6]$ and $\alpha+(32/97)\sqrt 6>1$. Then $\Theta'_1(x)< 0$ for $x\in[\alpha,1)$.

 Assume $x\in(0,\alpha)$ with $\alpha=(8/97)(11-2\sqrt 6)$, then
 $$\frac{\Theta'_1(x)}{K(x)^2}=-12(\Lambda(x)-\varphi_1(x))(\Lambda(x)-\varphi_2(x)),$$
 where $$\varphi_1(x)=\frac{1}{24}(16-x+\sqrt{97 x^2-176 x+64}),$$
 $$\varphi_2(x)=\frac{1}{24}(16-x-\sqrt{97 x^2-176 x+64}).$$
 On the one hand, $\varphi_2(x)\leq\varphi_1(x)$. Let us define on $(0,\alpha)$ the function
 \begin{equation}\label{ei}\psi(x)=24\frac{E(x)}{K(x)}-16+x-\sqrt{97 x^2-176 x+64}.\end{equation}
 First we check that for $x\in(0,\alpha)$
 \begin{equation}\label{ee}\frac{1}{24}(16-x+\sqrt{97 x^2-176 x+64})\leq \sqrt{1-x}+\frac{x^2}{16},\end{equation}
 Which is equivalent to
 $$(16-x-\frac32x^2)^2\leq (24\sqrt{1-x}-\sqrt{97 x^2-176 x+64})^2,$$
 or
 $$0\leq \frac{81 }{16}x^5+\frac{27 }{2}x^4-639 x^3+2376 x^2+23328 x+13824:=Q(x).$$
 and $x\in(0,\alpha)$. By successive differentiation we get
 $Q'(x)=\frac{405 x^4}{16}+54 x^3-1917 x^2+4752 x+23328$, $Q''(x)=(27/4)(15 x^3+24 x^2-568 x+704)$ and $Q'''(x)=(27/4)(45 x^2+48 x-568)>0$, $Q''(0)>0$ and then $Q'(x)>Q'(0)>0$, moreover $Q(0)>0$. Which implies equation \eqref{ee}. Therefore, from equation \eqref{ei}, we obtain
 $$\frac1{24}\psi(x)\geq \frac{E(x)}{K(x)}-\sqrt{1-x}-\frac{x^2}{16},$$
 which is positive by Lemma \ref{r}. Whence, for all $x\in(0,\alpha)$, $\Lambda(x)\geq\varphi_1(x)>\varphi_2(x)$ and then $\Theta_1'(x)>0$ for all $x\in(0,\alpha)$.

 This proves that $\Theta_1(x)$ is strictly decreasing. Furthermore, $\Theta_1(0)=0$, then $\varphi'(x)<0$ on $(0,1)$ and $\varphi(x)$ is strictly decreasing on $(0,1)$. Using the series expansion of the elliptic functions for $x$ close to $0$, we get
 $\varphi(x)=8/5 - (7 /50)x+o(x)$, and $\lim_{x\to 1^-}\varphi(x)=\log 4$.
\end{proof}
\begin{proposition}\label{r5} For $x\in[0,1)$, let $$u(x)=\frac1{16}{}_2F_1(3/2,3/2,3,x)(1-x)+\frac12{}_2F_1(1/2,1/2,2,x),$$
 $$v(x)=\frac12{}_2F_1(1/2,1/2,2,x)+{}_2F_1(1/2,1/2,1,x),$$
 and $$\Delta(x)=v(x)^2-4u(x){}_2F_1(1/2,1/2,1,x).$$

 Then, the function $u(x)$ is strictly increasing from $[0,1)$ onto $[9/16,2/\pi)$ and the function $\Delta(x)$ is strictly increasing from $(0,1)$ onto $(0,+\infty)$.

 The function $$w_+(x)=\frac12\log(1-x)+\frac{v(x)+\sqrt{\Delta(x)}}{2 u(x)},$$ initially defined on $[0,1)$, is extended to a continuous function on $[0,1]$ with $w_+(1)=\log 4$ and $w_+(0)=4/3$. It admits a maximum $a_c$ on $[0,1)$ and $a_c\geq \log 4$.

 The function $$w_-(x)=\frac12\log(1-x)+\frac{v(x)-\sqrt{\Delta(x)}}{2 u(x)},$$ is continuous on $[0,1)$ with $w_-(0)=4/3$ and $\displaystyle\lim_{x\to 1^-}w_-(x)=-\infty$.
\end{proposition}

\begin{proof} \
1)Differentiation and using Lemma \ref{lem1}, we get
$$u'(x)=(3/64)(1-x){}_2F_1(5/2,5/2,4,x).$$
Hence, $u(x)$ is strictly increasing on $(0,1)$. Furthermore, $u(0)=9/16$ and  by Lemma \ref{lem1}, $\displaystyle\lim_{x\to 1^-}u(x)=2/\pi$.

2)
On the one hand, we have
$$\begin{aligned}4\Delta(x)&={}_2F_1(1/2,1/2,2,x)^2+4{}_2F_1(1/2,1/2,1,x)^2
-{}_2F_1(3/2,3/2,3,x){}_2F_1(1/2,1/2,1,x)(1-x)\\&-4{}_2F_1(1/2,1/2,1,x){}_2F_1(1/2,1/2,2,x).\end{aligned}$$
Differentiation and using Lemma \ref{lem1}, it follows that
$$\begin{aligned}4\Delta'(x)=&2{}_2F_1(3/2,3/2,2,x){}_2F_1(1/2,1/2,1,x)+\frac12{}_2F_1(3/2,3/2,3,x){}_2F_1(1/2,1/2,1,x)\\
&-{}_2F_1(3/2,3/2,2,x){}_2F_1(1/2,1/2,2,x)-\frac3{4}{}_2F_1(3/2,3/2,4,x){}_2F_1(1/2,1/2,1,x)\end{aligned}.$$
Then, $4\Delta'(x)=\Delta_1(x){}_2F_1(3/2,3/2,2,x)+\Delta_2(x){}_2F_1(1/2,1/2,1,x),$
where $$\Delta_1(x)={}_2F_1(1/2,1/2,1,x)-{}_2F_1(1/2,1/2,2,x),$$
and $$\Delta_2(x)={}_2F_1(3/2,3/2,2,x)+\frac12{}_2F_1(3/2,3/2,3,x)-\frac34 {}_2F_1(3/2,3/2,4,x).$$

Using the series expansion of the hypergeometric function and the formula\\ $(a+1)_n=((a+n)/a)(a)_n$ we obtain
$$\Delta_1(x)=\sum_{n=0}^\infty\frac{((1/2)_n)^2}{(n!)^2}\frac{n}{n+1}x^n,$$
and
$$\Delta_2(x)=\sum_{n=0}^\infty\frac{((3/2)_n)^2}{n!(2)_n}\frac{(n+3)^2-9/2}{(n+3)(n+2)}\,x^n.$$
Then, $\Delta'(x)>0$ and $\Delta(x)$ is strictly increasing on $(0,1)$, furthermore, $\Delta(0)=0$
Moreover,
${}_2F_1(1/2,1/2,1,x)/v(x)\leq 1$, then $\Delta(x)\geq v(x)-4u(x),$ and from Lemma \ref{lem1}, we get $\displaystyle\lim_{x\to 1^-}v(x)=+\infty$.

3) We have $$w_+(x)=\frac12\log(1-x)+\frac{v(x)+\sqrt{\Delta(x)}}{2 u(x)},$$
then $$w_+(x)=\frac12\log(1-x)+\frac{v(x)}{2 u(x)}(1+(1-4u(x){}_2F_1(1/2,1/2,1,x)/(v(x))^2)^{1/2}).$$
Hence, $$w_+(x)=\frac12\log(1-x)+\frac{v(x)}{u(x)}-\frac{{}_2F_1(1/2,1/2,1,x)}{v(x)}+o\Big(\frac{{}_2F_1(1/2,1/2,1,x)}{v(x)}\Big).$$
Therefore,  $$w_+(x)=\frac12\log(1-x)+\pi2\,{}_2F_1(1/2,1/2,1,x)+o(1).$$
Using asymptotic formula in Lemma \ref{lem1}, we obtain  $w_+(x)=\log 4+O(1-x)$. Then, $\displaystyle \lim_{x\to 1^-}w_+(x)=\log 4$, furthermore, the function $w_+(x)$ is continuous on $[0, 1)$ with $w_+(0)=4/3$. Then $w_+(x)$ admits a maximum on $[0,1]$.

4) Recall that $$w_-(x)=\frac12\log(1-x)+\frac{v(x)-\sqrt{\Delta(x)}}{2 u(x)}.$$
As in the previous item, one checks that $w_-(x)$ is continuous on $[0,1)$ and $w_-(0)=4/3$. Furthermore,
   $$v(x)-\sqrt{\Delta(x)}=\frac{4u(x){}_2F_1(1/2,1/2,1,x)}{v(x)+\sqrt{\Delta(x)}}.$$
   Then, $$w_-(x)=\frac12\log(1-x)+\frac{2{}_2F_1(1/2,1/2,1,x)}{v(x)+\sqrt{\Delta(x)}}.$$
   Hence, $w_-(x)\leq \frac12\log(1-x)+2$ and then $\displaystyle\lim_{x\to 1^-}w_-(x)=-\infty$.
\end{proof}

\section{Proofs of the results}
\subsection{Proof of Theorem \ref{t2}}
\begin{proof}.
  From the properties of the hypergeometric function Lemma \ref{lem1}, we get
$$\begin{aligned}\frac4\pi f'_a(x)&=\frac{\frac12\,{}_2F_1(3/2,3/2,2,x)(a-\frac12\log(1-x))-\frac1{1-x}\,{}_2F_1(1/2,1/2,1,x)}{(a-\frac12\log(1-x))^2}\\
&=\frac{\frac12\,{}_2F_1(1/2,1/2,2,x)(a-\frac12\log(1-x))-{}_2F_1(1/2,1/2,1,x)}{(1-x)(a-\frac12\log(1-x))^2},\end{aligned}$$
and
$$\begin{aligned}\frac4\pi f''_a(x)&=\frac{\frac1{16}\,{}_2F_1(3/2,3/2,3,x)(a-\frac12\log(1-x))^2(1-x)}{(1-x)^2(a-\frac12\log(1-x))^3}
&=\frac{\frac1{16}\,{}_2F_1(3/2,3/2,3,x)(a-\frac12\log(1-x))^2(1-x)}{(1-x)^2(a-\frac12\log(1-x))^3}\\&+
\frac{\Big(a-1-\frac12\log(1-x)\Big)\Big(\frac12\,{}_2F_1(1/2,1/2,2,x)(a-\frac12\log(1-x))-{}_2F_1(1/2,1/2,1,x)\Big)}{(1-x)^2(a-\frac12\log(1-x))^3}.\end{aligned}$$
We set
$$g_a(x)=\frac4\pi f''_a(x) (1-x)^2(a-\frac12\log(1-x))^3 
.$$
Utilizing the notation from Proposit4on \ref{r5}, we express $g_a$ as
$$g_a(x)=(a-\frac12\log(1-x))^2u(x)-(a-\frac12\log(1-x))v(x)+{}_2F_1(1/2,1/2,1,x).$$
Therefore,
$$g_a(x)=u(x)\big(a-w_+(x)\big)\big(a-w_-(x)\big),$$
where $w_{\pm}(x)=\frac12\log(1-x)+\frac{v(x)\pm\sqrt{\Delta(x)}}{2 u(x)}.$
 Clearly, $w_+(x)> w_-(x)$ and $u(x)>0$ for all $x\in(0,1)$.  The function $f_a$ is strictly convex respectively concave on $(0,1)$ if and only if
$g_a(x)> 0$ respectively $g_a(x)< 0$ for all $x\in(0,1)$, a condition that is equivalent to $a\geq\displaystyle\max_{(0,1)}w_+(x)$ or $a\leq\displaystyle\inf_{(0,1)}w_-(x)$ respectively  $\min_{(x\in(0,1)}w_+(x)\geq a\geq\displaystyle\max_{(0,1)}w_-(x)$.
Applying Propositon \ref{r5} and using $\displaystyle\lim_{x\to 1^-}w_-(x)=-\infty$, we get
$f_a$ is is strictly convex respectively concave on $(0,1)$ if and only if $a\geq a_c$ respectively $a=4/3$.
This completes the proof of theorem \ref{t2}.
\end{proof}

\subsection{Proof of Theorem \ref{t3}}
\begin{proof} Let $u_a(x)=1/f_a(x)$. Upon differentiating, we obtain
$$u'_a(x)=\frac{-E(x) ( a-(1/2)\log (1-x))+K(x) (x+(1-x)(a-(1/2) \log (1-x)))}{2 (1-x) x K(x)^2},$$
 In carrying out an additional differentiation, we get
 $$u''_a(x)=\frac{2xK(x)\Big(K(x)-E(x)\Big)+\Big(x(1-x) K(x)^2-2K(x)E(x)+2E(x)^2\Big)\Big(a-\frac12\log (1-x)\Big)}{4 (1-x)^2 x^2 K(x)^3}.$$
 Setting $$v_a(x)=\frac{4 (1-x)^2 x^2 K(x)^3}{2K(x)E(x)-x(1-x) K(x)^2-2E(x)^2}u''_a(x),$$ we obtain
 $$v_a(x)=\varphi(x)-a,$$ where $\varphi$ is the function defined in Lemma \ref{r}.

 From Lemma \ref{r}, we obtain $v_a(x)> 0$ for all $x\in(0,1)$ if and only if $a\leq\displaystyle\min_{x\in(0,1)}\varphi(x)=\log(4)$ and $v_a(x)<0$ for all $x\in(0,1)$ if and only if $a\geq\displaystyle\max_{x\in(0,1)}\varphi(x)=8/5$.
This completes the proof of theorem \ref{t3}.

\end{proof}
\subsection{Proof of Theorem \ref{t1}}
 \begin{proof} Let $g(x)=\log h_p(x)=p\log(1-x)+\log K(x).$

Straightforward computations give
$$-g'(x)=\frac p {1-x}+\frac{(1-x)K(x)-E(x)}{ 2x(1-x)K(x)},$$
and
\begin{equation}\label{20}-(1-x)^2g''(x)=p +\frac{(2x-1)(1-x) K(x)^2-2 x K(x) E(x)+E(x)^2}{4 x^2 K(x)^2},\end{equation}
Let $$G(x)=\frac{(2x-1)(1-x) K(x)^2-2 x K(x) E(x)+E(x)^2}{4 x^2 K(x)^2},$$
then
$$\begin{aligned}&4 (1-x) x^3 K(x)^3G'(x)\\&=2 (1-x)^2 K(x)^3+x K(x) E(x)^2-(1-x) K(x)^2 E(x)-E(x)^3\\
&=E(x)^2(xK(x)-E(x))-(1-x)K(x)^2(E(x)-2(1-x)K(x))\\
&=E(x)^2(xK(x)-E(x))-(1-x)K(x)^2(xK(x)-E(x)+2E(x)-(2-x)K(x))\\
&=(xK(x)-E(x))(E(x)^2-(1-x)K(x)^2)+(1-x)K(x)^2((2-x)K(x)-2E(x))\\
.\end{aligned}$$
From  Lemma \ref{r} item 2), we get 
$$4 (1-x) x^3 K(x)^3G'(x)\geq \Big(E(x)^2-(1-x)K(x)^2\Big)\Big(xK(x)-E(x)+\frac2\pi(1-x)K(x)^2\Big).$$
Since, $$xK(x)-E(x)+(2/\pi)(1-x)K(x)^2=(1-x)K(x)(\frac2\pi K(x)-1)+K(x)-E(x)>0,$$
which follows from the inequalities $K(x)>\pi/2$ and $K(x)> E(x)$ for $x\in(0,1)$. Applying Lemma \ref{r3} we get $G'(x)>0$ for all $x\in(0,1)$. Consequently, the function $G(x)$ is strictly increasing on $(0,1)$.
From the series expansions of the functions $E(x)$ and $K(x)$, we get, $$G(x)=-\frac7{32}+\frac x{32}+o(x),$$  and $G(1)=0$.
The function $g$ is strictly concave respectively strictly convex on $(0,1)$ if and only if $-p-G(0)<0$ respectively $-p-G(1)>0$. An equivalent condition being $p\geq 7/32$ respectively $p\leq 0$.
This concludes the proof of Theorem \ref{t1}.
\end{proof}
\subsection{Proof of Corollary \ref{p1}}

\begin{proof}
1) If $p\leq 0$, then by Theorem \ref{t1} the function $h_p$ is log-convex and then $h_p$ is convex.

Next, assume $p\geq 3(2+\sqrt 2)/8=p_0$. Differentiate yields
$$h''_p(x)=\frac{  \big((4p^2-8p+3)x^2 +(4p-5)x+2\big)K(x)-2\big(2(p-1) x+1\big) E(x)}{4 x^2(1-x)^{2-p}},$$
Now set, $J_p(x)=\big((4p^2-8p+3)x^2 +(4p-5)x+2\big)K(x) -2(2(p-1) x+1) E(x).$
Differentiate $J_p$ with respect to $p$ yields
$$\frac{d}{dp}J_p(x)=4x\big(2(p-1)x+1\big)K(x)-4x E(x),$$
Since, $p>1$ then $(d/dp)J_p(x)>4x(K(x)-E(x))\geq 0$. Therefore, $J_p$ is a strictly increasing function of $p$ and
$J_p(x)\geq J_{p_0}(x)$. Moreover,
$$J'_{p_0}(x)=\Big(\frac38(1-2\sqrt 2)x^2+(\frac 3{\sqrt 2}-2)x+2)K(x)-\Big(\frac {3\sqrt 2-2}2x+2\Big)E(x).$$
Another differentiation gives
$$\frac{16}3(1-x)J''_{p_0}(x)=\left((10 \sqrt{2}-7) x-8 \sqrt{2}+8\right) E(x)-(1-x) \left((6 \sqrt{2}-3) x-8 \sqrt{2}+8\right) K(x).$$
Whence,
$$\frac{16}3(1-x)J''_{p_0}(x)=\left((10 \sqrt{2}-7) x-8 \sqrt{2}+8\right) E(x)-(1-x) \left((6 \sqrt{2}-3) x-8 \sqrt{2}+8\right) K(x).$$
and $$(\frac{16}3(1-x)J''_{p_0}(x))'=\frac{3}{2} (\left(8 \sqrt{2}-6\right) E(x)-\left(8 \sqrt{2}-6-(6 \sqrt{2}-3) x\right) K(x)),$$
since, $6 \sqrt{2}-3>8 \sqrt{2}-6$, then
and $$(\frac{16}3(1-x)J''_{p_0}(x))'=\frac{3}{2} (\left(8 \sqrt{2}-6\right) E(x)-\left(8 \sqrt{2}-6-(6 \sqrt{2}-3) x\right) K(x)),$$
and $$(\frac{16}3(1-x)J''_{p_0}(x))''=-\frac34\frac{(2 \sqrt{2}-3) E(x)-(6 \sqrt{2}-3) (1-x) K(x)}{1-x},$$
or
$$(\frac{16}3(1-x)J''_{p_0}(x))''=-\frac34\frac{(2 \sqrt{2}-3) (E(x)-(1-x) K(x))-4\sqrt{2} (1-x) K(x)}{1-x}\geq 0,$$
then the function $(\frac{16}3(1-x)J''_{p_0}(x))'$ increases and equal $0$ for $x=0$, therefore, the function $\frac{16}3(1-x)J''_{p_0}(x)$ is strictly increasing and equal $0$ for $x=0$. This implies that the function $J'_{p_0}$ is strictly increasing, moreover, $J'_{p_0}(0)=0$ and hence $J_{p_0}(x)$ is strictly increasing with  $J_{p_0}(0)=0$. Whence, $J_p(x)>0$ and $h_p$ is strictly convex.

For the converse, one computes the limits
\begin{equation}\label{l1}\lim_{x\to 1} \frac{(1-x)^{2-p}h''_p(x)}{K(x)}=4p(p-1),\end{equation}
 and
\begin{equation}\label{l2}\lim_{x\to 0} (1-x)^{2-p}h''_p(x)=\frac{\pi}{64}   (32 p^2-48 p+9).\end{equation}
If $h_p$ is convex then $p(p-1)\geq 0$ and $32 p^2-48 p+9\geq 0$. Which gives $p\leq 0$ or $p\geq 3(2+\sqrt 2)/8$.

2) Now assume that $p\in[3(2-\sqrt 2)/8,1]$. Recall that
$$S_p(x)=\frac{d}{dp}J_p(x)=4x\big(2(p-1)x+1\big)K(x)-4x E(x),$$
Then, the function  $p\mapsto S_p(x)$ is strictly increasing, Moreover, $S_1(x)>0$ and for $p_1=3(2-\sqrt 2)/8$, $S_{p_1}(x)=4x\big(2(p_1-1)x+1\big)K(x)-4x E(x)<0$. So, $J_p(x)\leq \max(J_1(x), J_{p_1}(x))$.

On the one hand, $J'_1(x)=(1/2)(E(x)-(3 x+1) K(x))<0.$
Then $J_1(x)$ is strictly decreasing and we have $h_1(0)=0$. Therefore, $J_1(x)<0$. On the other hand,
$$J'_{p_1}(x)=-\frac3{16}\frac{\Big((7+10 \sqrt{2}) x-8(1+\sqrt{2})\Big)E(x)-(1-x) \Big((3+6 \sqrt{2}) x-8(1+\sqrt{2})\Big) K(x) }{(1-x)},$$
$$-\frac{16}3\Big((1-x)J'_{p_1}(x)\Big)'=\frac{3}{2} ((6+8 \sqrt{2}) E(x)+((3+6 \sqrt{2}) x-8 \sqrt{2}-6) K(x)),$$
$$\frac{16}3\Big((1-x)J'_{p_1}(x)\Big)''=\frac34\frac{ (3+2 \sqrt{2}) E(x)-3 (1+2 \sqrt{2}) (1-x) K(x)}{1-x},$$
$$\Big((1-x)\frac{16}3\big((1-x)J'_{p_1}(x)\big)''\Big)'=\frac3{2x}(4 \sqrt{2} E(x)-((3+6 \sqrt{2}) x+4 \sqrt{2}) K(x))<0.$$
Then, $\big((1-x)J'_{p_1}(x)\big)''<0$ and $\frac{16}3\Big((1-x)J'_{p_1}(x)\Big)'<0$. Thus, the function $\Theta_{p_1}(x)=(1-x)J'_{p_1}(x)$ is strictly decreasing, moreover, $\Theta_{p_1}(0)=0$, then $J_{p_1}(x)$ is strictly decreasing and $J_{p_1}(x)<h_{p_1}(0)=0$. Then $J_{p_1}(x)$ is strictly decreasing and since $J_{p_1}(0)=0$. Therefore, $J_p(x)<0$ and
$h''_p(x)<0$ for $x\in(0,1)$.

If $h_p$ is strictly concave, then by using the limits \eqref{l1} and \eqref{l2}, we get $p(p-1)\leq 0$ and $32 p^2-48 p+9\leq 0$. Which implies that $p\in[3(2-\sqrt 2)/8, 1]$.
This completes the proof.

\end{proof}


\subsection{Proof of Corollary \ref{p}}
\begin{proof}
1) Differentiation yields, $$h'_p(x)=\frac{(1-x)^{p-1}}{2 x}(E(x)+((1-2 p) x-1)K(x)).$$
If $h_p$ is strictly decreasing, then $h'_p(x)<0$ for all $x\in(0,1)$ and then
 $E(x)+((1-2 p) x-1)K(x)$, which is equivalent to \begin{equation}\label{O}\frac12-\frac{K(x)-E(x)}{2x K(x)}<p.\end{equation}
 Since, $\lim_{x\to 0}\frac{K(x)-E(x)}{x}=\pi/4$ and $K(0)=\pi/2$. From the equation above, we get $p\geq 1/4$.

 Next, assume $p\geq 1/4$. Then, $p\geq 7/32$, Applying Theorem \ref{t1}, the function $\log h_p(x)$ is strictly concave and the function $\big(\log h_p\big)'$ is strictly decreasing. Since, $\big(\log h_p\big)'(0)=1/4-p$. Therefore, $\big(\log h_p\big)'(x)<1/4-p<0$ and the conclusion follows.

 2) If $h_p$ is strictly increasing then equation \eqref{O} is reversed and by letting $x\to1-$, we get $p\leq 0$.

 Next, if $p< 0$ then the function $x\mapsto (1-x)^p$ is strictly increasing and positive and the function $K(x)$ is strictly increasing and positive, then $h_p$ is strictly increasing.

 3) Let $L_p(x)=E(x)+((1-2 p) x-1)K(x)$. Conducting a successive differentiation, we get
 $2(1-x)L'_p(x)=(1-2p)(1-x)K(x)-2 p E(x),$ and
 $$4x((1-x)L'_p(x))'=K(x) ((2 p-1)x+4p-1)+(1-4 p) E(x)\leq (4p-1)(K(x)-E(x))<0,$$
 for every $p<1/4$, then the function $(1-x)L'_p(x)$ is strictly decreasing on $(0,1)$ and equal $-p$ for $x=1$ and equal $(1-4p)\pi/2$ for $x=0$. Therefore, for $p\in(0,1/4)$ there exits a unique $y_p\in(0,1)$ such that $L_p$ is strictly increasing on $(0,y_p)$ and strictly decreasing on $(y_p,1)$. Moreover,
 $L_p(0)=0$, then there exists a unique $x_p\in(y_p,1)$ such that $L_p$ is positive on $(0,x_p)$ and negative on $(x_p,1)$. It follows that the function $f_p$ is strictly increasing on $(0,x_p)$ and strictly decreasing on $(x_p,1)$.

\end{proof}
\section{Inequalities}
As an immediate consequence of the monotonicity and concavity properties of $f_a$, we obtain a chain of mean value
inequalities,
$$\sqrt{f_p(x)f_p(y)}\leq f_p(\frac{x+y}2)\leq\frac{f_p(x)+f_p(y)}2\leq f_p(\sqrt{xy}),$$
for all $x,y\in(0,1)$.  The first inequality is valid for $p\in\Bbb R$, the second for $a\in[3(2-\sqrt 2)/8,1]$, the third for $a\geq 1/4$. Equality holds if and only if $x = y$. One deduces the following corollary.
\begin{corollary} For all $x\in(0,1)$ and $p\in[1/4,1]$,
$$4K(x)K(1-x)(x-x^2)^p\leq (x^pK(x)+(1-x)^pK(1-x))^2 \leq\alpha\leq 4(1-\sqrt{x-x^2})^{2p}K(\sqrt{x-x^2})^2,$$
where, $\alpha=\Gamma(1/4)^4/(2^{2+2p}\pi)$.
\end{corollary}
\begin{corollary} \
\begin{enumerate} \item Let $a\geq a_c$. For all $r\in(0,1)$, we have
$$\frac{4K(1/2)}{2a+\log 2}\leq \frac{K(r)}{a-\frac12\log (r)}+\frac{K(1-r)}{a-\frac12\log (1-r)}<1+\frac\pi{2a} .$$
Both bounds are sharp. The sign of equality holds if and only if $r=1/2$.

\item Let $p\geq 3(2+\sqrt 2)/8$. For all $r\in(0,1)$, we have
$$\frac{K(1/2)}{2^{p-1}}\leq r^pK(1-r)+(1-r)^pK(r)<\frac\pi2.$$
The inequality is reversed for $p\in[3(2-\sqrt 2)/8,1]$.
\item Let $p\geq 0$. For all $r\in(0,1)$, we have
$$2^{1+p}K(1/2)(r-r^2)^p\leq (1-r)^pK(1-r)+r^pK(r).$$
\item Let $p\geq 7/32$. For all $r\in(0,1)$, we have
$$\sqrt{(r-r^2)^pK(1-r)K(r)}\leq \frac{K(1/2)}{2^p}.$$
\end{enumerate}
\end{corollary}
\begin{proof} For $a\in\Bbb R$ and $x\in(0,1)$, let $H(x)=f(x)+f(1-x)$. Then
$$H'_a(x)=f'_a(x)-f'_a(1-x),\quad H_a''(x)=f''_a(x)+f''_a(1-x),$$
and $H_a(1/2)=0$.

1) For $a\geq a_c$ and $f=f_a$, $H''_a(x)>0$. it follows that $H_a$ is strictly decreasing on $(0, 1/2]$ and
strictly increasing on $[1/2, 1)$. Which implies that
$$H_a(\frac12)\leq H_a(x)<\min(H_a(0),H_a(1)),$$
with equality only if $x=1/2$. Since, $f_a(0)=\pi/(2a)$ and $f_a(1)=1$. Furthermore, $H_a(0)=H_a(1)=1+\pi/(2a)$ and $H_1(1/2)=K(1/2)/(2a+\log 2)$.

2) Let $p\geq 3(2+\sqrt 2)/8$ and $f=h_p$, then $ H''_p(x)>0$.
Therefore, $$H_a(\frac12)\leq H_p(x)<(H_a(0),H_a(1)) .$$
Since, $H_p(0)=H_p(1)=\pi/2$ and $H_p(1/2)=K(1/2)/2^{p-1}$.

3) Let $p\geq 0$, and $g_p(x)=h_{-p}(x)$, then from Proposition \ref{p1} $g_p$ is convex. For $f=g_p$ the function $H''_p>0$ and the result follows.

4) For $p\geq 7/32$ and $f=\log h_p$, $H''_p(x)<0$ and $H_p(1/2)=2\log(K(1/2)/2^p)$.
and $H_p(r)\leq H_p(1/2)$.
\end{proof}
Combine 2) and 4) and use the value $K(1/2)=\pi\sqrt\pi/(2\Gamma(3/4)^2)$, we get the double inequalities in the abstract.
\begin{corollary} Let $p\geq 1/4$. For all $r\in(0,1)$
$$\frac{\pi}2(1-r)^p<K(r)<\frac\pi{2(1-r)^p}.$$

Let $p\in(0,1/4)$. For all $r\in (0,x_p)$
$$\frac{\pi}{2(1-r)^p}<K(r)<\frac{(1-x_p)^pK(x_p)}{(1-r)^p},$$
where $x_p$ is the unique zero in $(0,1)$ of the equation $E(x)+((1-2p)x-1)K(x)=0$

\end{corollary}
The proof of the corollary follows from the monotonicity properties of the function $h_p$ of Proposition \ref{p}.


\begin{thebibliography}{9}
\bibitem{alz}  {Horst Alzer and Song-Liang Qiu},
        {Monotonicity theorems and inequalities for the complete elliptic integrals},
       {J. Comput. Appl. Math},
          {2004},
          {172},
         {2},
          {289-312}

\bibitem{alz1} {Horst Alzer  and Kendall Clyde Richards },
         {A concavity property of the complete elliptic integral of the
first kind},
        {Integral. Transforms. Spec. funct.},
          {2020},
         {1-11}
                   
\bibitem{ande} {Glen Douglas Anderson  and Song-Liang Qiu and Mavina Krishna Vamanamurthy},
        {Elliptic integral in equalities, with applications},
       {Constr. Approx},
         {1998},
        {14},
      {2},
       {195-207}

\bibitem{ander} {Glen Douglas Anderson and Mavina Krishna Vamanamurthy and Matti Vuorinen},
         {Functional inequalities for complete elliptic
integrals and their ratios},
        {SIAM J Math Anal},
          {1990},
         {21},
       {536-549}
  \bibitem{vuo}  {Glen Douglas Anderson and Mavina Krishna Vamanamurthy and Matti Vuorinen},
         {Conformal Invariants, Inequalities, and Quasiconformal Maps},
      {John Wiley \& Sons},
          {1997},
       {New York}
  \bibitem{vam} {Glen Douglas Anderson and Mavina Krishna Vamanamurthy and Matti Vuorinen},
          {Inequalities for quasiconformal mappings in space},
        {Pacific. J. Math},
           {1993},
         {160},
          {1-18}      
\bibitem{ask}  {George E. Andrews and Richard Askey and Ranjan Roy},
        {Special Functions},
     {Cambridge University Press},
         {1999},
        {17},
       {Cambridge},
      {Encyclopedia of Mathematics and its Applications}

 \bibitem{bal} {Imran Abbas Baloch and Yu-Ming Chu},
     {Petrovi-Type inequalities for harmonic $h-$convex functions},
      {Journal of Function Spaces},
       {2020},
     {1},
    {1-7}
 
 \bibitem{bier}   {Mieczyslaw Biernacki and Jan Grzegorz Krzy\'z},
      {On the monotonicity of certain functionals in the theory of analytic
functions},
    {Ann. Univ. MariaeCurie-Sk lodowska Sect. A},
         {1955},
        {9},
      {135-147}
         

     
\bibitem{ber} { Bruce Carl Berndt},
    {Ramanujan's Notebooks, Part II},
   {Springer-Verlag},
        {1989},
        {New York}

  \bibitem{bha} {Bruce Carl Berndt and Srinivasamurthy Bhargava and Frank G. Garvan},
     {Ramanujan's theories of elliptic functions to alternative bases},
     {Trans. Amer. Math. Soc},
     {1995},
    {347},
      {11},
       {4163-4244}

\bibitem{byr}
  {Paul Francis Byrd and Morris David Friedman},
      {Handbook of Elliptic Integrals for Engineers and Scientists},
     {Springer-Verlag},
        {1971},
     {New York}
             
\bibitem{chen}  {Ya-jun Chen and Tie-hong Zhao},
  TITLE =        {On the monotonicity and convexity for generalized elliptic
integral of the first kind},
     {Rev. Real Acad. Cienc. Exactas Fis. Nat. Ser. A-Mat},
     {2022},
          {116},
        {77}






\bibitem{koum1}  {Stamatis Koumandos and Henrik L. Pedersen},
     {On the asymptotic expansion of the logarithm of Barnes
triple gamma function},
     {Math. Scand},
      {2009},
      {105},
      {2},
      {287-306}.
  \bibitem{olv}  {Frank W. Olver and Daniel Lozier and Ronald F. Boisvert and and Charles Winthrop Clark},
         {NIST handbook of mathematical functions},
      {Cambridge University Press},
          {2010},
       {Cambridge (UK) }
       
     \bibitem{ponn}  {Saminathan Ponnusamy and Matti Vuorinen},
      {Asymptotic expansions and inequalities for hypergeometric
functions},
     {Mathematika},
      {1997},
       {44},
        {2},
       {278-301}


\bibitem{qiu} {Song-Liang Qiu and Mavina Krishna Vamanamurthy and Matti Vuorinen},
      {Some inequalities for the growth of elliptic integrals},
      {SIAM J. Math. Anal},
       {1998},
      {29},
       {5},
        {1224-1237}

\bibitem{hua}   { Song-Liang Qiu and Xiao-Yan Ma and Ti-Ren Huang },
        {Sharp Approximations for the Ramanujan Constant},
     {Constructive Approximation},
         {2019},
       {51},
         {303-330}
         
\bibitem{hua}   { Song-Liang Qiu and Xiao-Yan Ma and Ti-Ren Huang },
        {Sharp Approximations for the Ramanujan Constant},
     {Constructive Approximation},
         {2019},
       {51},
         {303-330}
  \bibitem{rich}  { Kendall Clyde Richards and Jordan N Smith },
         {A concavity property of generalized complete
elliptic integrals},
        {Integral. Transforms. Spec. funct.},
          {2020}        
          
           \bibitem{wan}  {Miao-Kun Wang and Hong-Hu Chu and Yong-Min Li and Yu-Ming Chu},
       {Answers to three conjectures on the convexity of three
functions involving complete elliptic integrals of the first kind},
    {Appl. Anal. Discrete Math},
   {2020},
    {14},
        {255-271}
        
\bibitem{wang}   {Fei Wang and Bai-Ni Guo and Feng Qi},
      {Monotonicity and inequalities related to complete elliptic integrals of the second kind},
      {AIMS Math},
          {2020},
     {5},
      {3},
        {2732-2742}

  

\bibitem{yang} {Zhen-Hang Yang and Yu-Ming Chu and Miao-Kun Wang},
      {Monotonicity criterion for the quotient of power
series with applications},
       {J. Math. Anal. Appl},
          {2015},
        {428},
       {1},
         {587-604}

 
      
\bibitem{yan} {Zhen-Hang Yang and JingFeng Tian},
         {Convexity and monotonicity for elliptic integrals of the first kind and
applications},
        {Appl. Anal. Discrete. Math},
          {2019},
         {13},
         {1},
       {240-260}        
\end{thebibliography}
\end{document}